\documentclass[graybox]{svmult}
\usepackage{graphicx}
\usepackage{multicol}
\usepackage{multirow}
\usepackage{dcolumn}
\usepackage[figuresright]{rotating}
\usepackage{amssymb}
\usepackage{amsfonts}
\usepackage{amsmath}
\usepackage{amscd}
\usepackage{amsbsy}
\usepackage{mathrsfs}
\usepackage{bbm}
\usepackage{subfigure}
\usepackage[OT2,OT1,T1]{fontenc}
\usepackage{algorithm} 
\usepackage{algorithmicx} 
\usepackage{algpseudocode} 
\usepackage[normalem]{ulem}
\usepackage{dsfont}
\usepackage{srcltx}

\makeindex
\spnewtheorem*{questi}{Question}{\upshape\bfseries}{\itshape}
\spnewtheorem*{probl}{Problem}{\upshape\bfseries}{\itshape}
\spnewtheorem{remarks}{Remarks}{\itshape}{\rmfamily}
\spnewtheorem*{ques}{Question}{\bfseries}{\rmfamily}
\spnewtheorem*{quest}{Questions}{\upshape\bfseries}{\itshape}
\spnewtheorem{exper}{Experiment}{\bfseries}{\rmfamily}
\spnewtheorem*{defin}{Definition}{\bfseries}{\rmfamily}
\spnewtheorem*{theo}{Theorem}{\upshape\bfseries}{\itshape}
\spnewtheorem*{rem}{Remark}{\itshape}{\rmfamily}
\spnewtheorem{conject}{Conjecture}{\bfseries}{\itshape}

\newcommand{\ams}[1][]{\par\addvspace\baselineskip
 \textbf{AMS subject classification:}\enspace\ignorespaces}
\newcommand\cyr{
\renewcommand\rmdefault{wncyr}%
\renewcommand\sfdefault{wncyss}%
\renewcommand\encodingdefault{OT2}%
\normalfont \selectfont}
\DeclareTextFontCommand{\textcyr}{\cyr}
\allowdisplaybreaks
\makeatletter
\newcolumntype{d}[1]{D{.}{.}{#1}}
\makeatother
\newcommand\checkm{{^\checkmark}}  

\begin{document}

\setcounter{chapter}{17}

\title{A Distributional Approach to Generalized Stochastic Processes on Locally Compact Abelian Groups\label{chap18}}

\authorrunning{A Distributional Approach to Generalized Stochastic Processes on Locally Compact Abelian Groups}

\author{H.G. Feichtinger and W. H\"ormann}

\authorrunning{H.G. Feichtinger and W. H\"ormann}

\institute{H.G. Feichtinger \at Faculty of Mathematics, University of Vienna \\Oskar-Morgenstern-Platz 1 1090 Vienna, Austria \and
W. H{\"o}rmann \at Department of Industrial Engineering, Bogazici University,\\ 34342 Bebek, Istanbul, Turkey}

\maketitle

\section*{Dedication}

This paper is dedicated to \textit{Paul Butzer
at the occasion of his 85th birthday}.
His example and work has strongly
influenced not only the first author of this paper but
generations of mathematicians, working in approximation
theory and Fourier analysis. He also showed the younger colleagues
how important it is to be open for applied areas, not to have a
narrow scope of work, and explore different ways of understanding
and seeing mathematical facts. Several of his papers concern the
logical equivalence of fundamental facts in analysis (take
just \cite{burist84} as an example).

It is perhaps not so well known
to many familiar with his role as a leading figure in approximation
theory that he has also done significant work in probability theory (\cite{buhawe75,buha79}).
So I hope that Paul will enjoy reading this note, which indicates
that a purely~functional analytic approach
to generalized stochastic processes based on the Segal
algebra ${S_{0}(G)}$  is possible which avoids many of the technical problems associated with the usual presentation, based on vector-valued integrals
or topological vector space theory.

The material of this note
is mainly based on the PhD thesis of the second author (\cite{ho89})
and was available (only) as a technical report already for a long time.
Rising recognition of the universal usefulness of this space of
test functions (see \cite{gr01,cofelu08,fe09})
and the surrounding family of modulation spaces
has encouraged the authors to make the manuscript available
in this way to the community. In addition to the publications \cite{ke03,wa05,wa06-2} (all related to this old technical
report)  there is an upcoming series of papers by G.~Pfander
(\cite{pfwa05-1,pfzhXX-1,pfzhXX,pfzhXX-B,okpfzh11})
which is exploring in great detail
the methods and settings described in the current paper.
They develop a theory of stochastic modulation spaces, in analogy to classical (deterministic) modulation space (\cite{fe06}) in order to enable sampling of operators theory
(\cite{pf13-1,pfwaXX,kopf06})
to be applied to stochastic operators.

\section{Introduction}

Whereas ordinary functions on a locally compact group $G$ map the
group elements into the complex numbers, a stochastic process
can be understood as a mapping into a Hilbert space $\mathcal{H}$,
namely into the Hilbert space  $\mathcal{H} =   L^2_0(\varOmega,\varSigma,P),$
consisting of square integrable random variables $X(\omega)$
with expected value zero $E(X) = 0$.

The idea of a generalized function is to reduce the
knowledge about the function to that of certain averages.
It leads to the
concept of generalized functions as continuous
linear functionals on spaces
of test functions. The combination of both ideas is
the basis for generalized
stochastic processes on locally compact Abelian
(in short: LCA) groups:
Hilbert space valued bounded linear operators on spaces
of test functions on G. The properties of
the Schwartz space $\mathcal{S}(\mathbb{R}^d)$,
in particular its invariance under the
Fourier transform (defined on $L^1(\mathbb{R}^d)$),
make it a very suitable tool for the description of
generalized functions, but its generalization, the so-called
Schwartz-Bruhat space, is very complicated
(cf. \cite{re89}) and structure theory of LCA
groups is required to describe the space.

Summarized by a diagram we have the following situation:

\begin{table}
{\begin{tabular*}{237pt}{l|l@{\ }}
\hline
\multicolumn{2}{l}{From ordinary functions to generalized stochastic processes} \\
\hline
Ordinary functions on $\mathbb{R}^d$  & Stochastic processes on $\mathbb{R}^d$ \\
\hline
$\mathbb{R}^d \to \mathbb{C} $  &
$\mathbb{R}^d \to {\mathcal{H}}  $ \\
\hline
$t \mapsto f(t)$ & $t \mapsto X(t,\omega) \in  {L^2(\varOmega,\varSigma,P)} $ \\
\hline
\textit{Generalized functions}  & \textit{Generalized stochastic processes} \\
\hline
${ C_c(\mathbb{R}^d)} \to \mathbb{C} $ & $\rho: { S_0(\mathbb{R}^d)} \mapsto {\mathcal{H}}: f \mapsto \rho(f)$   \\
\hline
$k \mapsto  \sigma(f):= \int_{\mathbb{R}^d} k(x) f(x) dx $  & e.g.\ $ f \mapsto \int_{\mathbb{R}^d} X(t,\omega) k(t) dt $\\
\hline
\end{tabular*}}
{}
\end{table}

In contrast, the function space ${S_{0}(G)}$ discovered by the
first author (cf.
\cite{fe81-2,re89}) can be defined without the use of structure
theory in a simple
way for general LCA groups. Moreover it is a Banach space (which greatly
simplifies the description of the natural topology on the dual space), and
the Fourier transform maps ${S_{0}(G)}$ onto $S_{0}(\hat{G})$, the
corresponding space on the dual group $\hat{G}$, which consists
 of continuous characters $\chi$ on $G$, i.e.,  continuous
 group homomorphism from $G$ to the torus group $T = \{ z | \, |z| = 1\}
 \subset \mathbb{C}$.  Using Pontryagin's
duality theorem, it is then possible to extend the Fourier transform in order
to obtain a generalized Fourier transform from
${S_0{ '}(G)}$ onto $S_{0}'(\hat{G})$. Since the
space $\mathcal{S}(G)$ of Schwartz-Bruhat is dense in ${S_{0}(G)}$
it is clear that the
concept of tempered distributions $\sigma \in \mathcal{S}'(G)$ is more
general. However, if one is not
interested in derivatives, the concept of ${S_0{ '}(G)}$ is general enough
and has many
practical advantages, mainly due to its simplicity. Let us only mention that
for the case $G= \mathbb{T}$ the space ${S_{0}(G)}$ coincides with $A(T)$,
Wiener's algebra of absolutely convergent Fourier series,
of which it can be seen as a natural generalization for arbitrary LCA groups.

The concept of generalized stochastic processes over Euclidean spaces
has been developed in \cite{ge55,gevi64} and \cite{it54} %
where the space $\mathcal{D}$ of infinitely often differentiable functions
with compact support is used.  In \cite{ni75} the space $\mathcal{C}_c(G)$
of continuous functions with compact support serves as the space of test
functions.  The main disadvantage of these function spaces (besides
technical questions that may be overcome by developing the appropriate
integration or distribution theory) is in our opinion the fact that they
are not invariant under the Fourier transform and that they are only
topological vector spaces.  The only work on generalized stochastic
processes we know that uses a test function space that is invariant
under Fourier transform is \cite{ja79-1} for the case $G = \mathbb{R}$.
There the function space is defined over $\mathbb{C}$ or with technical difficulties over
$\mathbb{C}^n$, but it seems impossible to extend this definition to locally
compact Abelian groups. In addition, the test functions are analytic
functions   making it impossible to define the  support
for elements of  the corresponding dual space. Another alternative
approach is given in \cite{ob95}.

The observation that ${S_{0}(G)}$ is a conceptually and technically much more convenient space and that most of the relevant concepts arising in the theory of generalized stochastic processes over LCA  groups can be proved on the basis of this concept lead to this paper.  The interested reader may
find more details about generalized stochastic processes and, in
particular, the definition of the Wigner distribution for a process over
$\mathbb{R}^n$ in the thesis of the second author \cite{ho89}.

We do not claim to provide completely new results here which
cannot be found (using mostly a slightly different setting) in the literature.
However we think that our approach allows to derive and teach
such results much better and easier.
We also do not claim that our approach provides the
most general approach. In fact, the setting
based on the Schwartz(-Bruhat) space is more general,
but requires a  heavier machinery,
at least for the setting of LCA groups.

\section{Notations}

Let $G$ be a locally compact Abelian group with Haar measure $dx$. The group
operation is written as addition.  The dual group is denoted by
$\hat{G}$; $\chi_x$ denotes a character on $\hat{G}$ which can be
identified with $x \in G$. We write $C^{b}(G), C_{0}(G)$, and
$\mathcal{C}_c(G)$ 
for the spaces of continuous complex-valued functions which are
bounded, with limit zero at infinity and with compact support,
respectively. $C^{b}(G)$ and $C_{0}(G)$ are endowed with the supremum
norm $\| . \| _{\infty }$, and $\mathcal{C}_c(G)$ with the inductive limit
topology is continuously embedded into $C_{0}(G)$ as a dense
subspace.

\looseness-1 In our context, $M(G)$ will be simply considered as the Banach dual
of $C_{0}(G)$ (with the sup-norm $\| . \| _{\infty }$).
By the Riesz representation
theorem, it can be identified with the set of all bounded and
regular Borel measures with the variation norm and contains
$L^1(G)$ as closed subspace (the absolutely continuous measures).
However, for our purpose, it is more convenient to just use the
characterization of this norm via the functional norm
$ \|\mu\|_M = \sup_{\|f\|_{\infty} \leq 1} |\mu(f)|$.
The \textit{translation operator} $T_{x}$ and the
\textit{modulation operator} resp.\ the multiplication (by characters) $M_{\chi}$ are defined~by
\[
T_{x}f(y) := f(y-x), \, x \in G,
\]
\[
M_{\chi}f(y) := \chi(y)f(y), \, \chi \in \hat{G}.
\]
Over $G = R^d$ we use of course the usual pure frequencies
$\chi_s(y) = exp( 2 \pi i \langle s, y \rangle)$ and identify
$\widehat{R^d}$ with $R^d$ in this way, and write $M_s$ for
$M_{\chi_s}$.

The central object of this manuscript will be the Segal
algebra ${S_{0}(G)}$ introduced in \cite{fe81-2},
 which will be shown to be a very good substitute for the
 Schwartz space of rapidly decreasing functions for the
 setting discussed in this paper. This space has
 meanwhile gained popularity within TF analysis
 (time-frequency analysis), where it is denoted by
 $M^1(G)$ (see \cite{gr01}, Chaps.~11 and~12).
For us only a couple of basic properties of ${S_{0}(G)}$,
which has been defined for general LCA groups, will
play a role, but with the hint that these properties
are characteristic (see \cite{lo83-1}).

$S_{0}(G)$ is a subspace of the Fourier algebra $A(G)$,
defined as the image of $L^1(\hat{G})$ under the (inverse)
Fourier transform, with $\|h\|_{A(G)} = \|f\|_1$
for $ \hat{h} = f$.
We note $A_c(G) := C_c(G) \cap A(G)$
 is a dense subspace of $A(G)$.

\begin{defin}
Let  a nonzero function $k\in A(G)\cap \mathcal{C}_c(G)$, i.e.,
$k$ is continuous, with compact support and integrable
Fourier transform $\hat k$,  be arbitrary but fixed.
Then the Segal\footnote{Also called
Feichtinger's algebra in the literature, see, e.g.\
 \cite{rest00}.}
 algebra  ${S_{0}(G)}$ is defined by
\[
{S_{0}(G)}:=\{f \in A(G)\mid \; \| f\| _{S_0}:= \int_G \| T_y k\cdot f\|_A\:dy <
\infty \}.
\]
\end{defin}

Different functions $k_1$ and $k_2$ as above define the
same space with equivalent norms. Another way to describe
this definition is to say that ${S_{0}(G)}$ is the Wiener amalgam
space $W(A,L^1)(G)$ (see \cite{fe83}). Further alternative
characterizations are described below (e.g.\ replacing $k$
by a Gauss function).

There are many sufficient conditions which imply that a function
belongs to the Segal algebra ${S_{0}(G)}$, especially for $G = \mathbb{R}^d$.
The paper
\cite{gr96}
gives sufficient conditions of the form of a combined decay in time
and frequency, using weighted $L^p$-spaces, with polynomial
weights, on both sides.

\cite{fewe06} demonstrates that essentially all known
\textit{classical summability kernels}  belong to ${ S_0(\mathbb{R}^d)}$.
In most cases this is shown by demonstrating that they
belong to the subclass ${{V}}^2_1(\mathbb{R}^d)$.

All these results easily imply that
$\mathcal{S}(\mathbb{R}^d) \subset { S_0(\mathbb{R}^d)}$, which was first demonstrated by
Poguntke in \cite{po80-1}.

${S_{0}(G)}$ can be shown to be the smallest
Banach space among all Banach spaces which are isometrically
invariant under translation and character multiplication
and it contains all $f\in L^1(G)$ with compactly supported
Fourier transform. Another property of ${S_{0}(G)}$---the
invariance under Fourier transform---is
 very important for this work.

For the case $G =  \mathbb R^d$ a more convenient (equivalent)
characterization (more in the spirit of a TF-object) can be
given, using the short-time Fourier transform:

\begin{defin}
Let $f,g \in L^2(\mathbb R^d)$. The \textbf{short-time Fourier transform}

(STFT) $\mathcal{V}_g(f):
{\mathbb R^d}\times {\hat{\mathbb R}}^d\longrightarrow {\mathbb C}$
of $f$ with window $g$ is defined by
\begin{equation}\label{STFT}
(\mathcal{V}_gf)(x,\xi ):=\int_{\mathbb R^d}f(t){\overline{g(t-x)}}
e^{-2\pi i\xi \cdot t}dt = \langle f,M_\xi T_x g \rangle.
\end{equation}
\end{defin}
Using the STFT one has the following alternative definition:

\begin{defin}
Let $g_0$ be the Gauss function, i.e.,
$g_0(x):=e^{-\pi |x|^2}$.
$S_0(\mathbb R^d)$ is then given by
\begin{equation}\label{S0STFT}
S_0(\mathbb R^d):=\{f\in L^2(\mathbb R^d)\mid \|f\|_{S_0}=
\|\mathcal{V}_{g_0}f\|_{L^1(\mathbb R^{2d})}<\infty \}.
\end{equation}
\end{defin}

The Banach space ${S_{0}(G)}$
has been characterized as the smallest Banach space among all Banach
spaces which are isometrically invariant under translation and character
multiplication and containing all integrable functions $f\in L^{1}(G)$
with compactly supported Fourier transform (cf.  \cite{fe81-2} for
details).  For a summary of properties of this space (such as invariance
under the Fourier transform) the reader is referred to \cite{fe81-2} or
the survey article \cite{fe89-2}.  We will work mostly with the elements
of ${S_0{ '}(G)},$ which will be called \textbf{distributions}.
This terminology is justified because one can show that the
Schwartz space (resp.\ in general the Schwartz-Bruhat space)
is embedded into ${S_{0}(G)}$ as a dense subspace and consequently
${S_0{ '}(G)}$ is a (dual) Banach space of tempered distributions.
Note however that it is one of the main purposes of this article
to point out that the theory of GSPs can be built solely using
the Banach space ${S_{0}(G)}$ and its dual, without any recurrence
to the theory of tempered distribution, nor even topological
vector spaces or Lebesgue integration theory.

For convenience we shall write for $\sigma \in {S_0{ '}(G)}$ and
$f\in {S_{0}(G)}: \langle \sigma ,f\rangle := \sigma (f)$ for
the natural bilinear pairing of elements from dual Banach spaces.

Many of the actions which one can define on ${S_{0}(G)}$
can be extended to ${S_0{ '}(G)}$ in a natural way, e.g.\
the Fourier transform, multiplication with functions or
convolutions. In principle there are two ways for such
extensions. One way to define the unique extension
of operators well defined on ${S_{0}(G)}$ into ${S_0{ '}(G)}$
is to make use of the  ${w}^{\ast}$-density
of ${S_{0}(G)}$ within ${S_0{ '}(G)}$, but this approach---although
looking more elementary---is more cumbersome in practice.
Alternatively one can use of the duality of operators
(which is particularly attractive for operators which also
act as unitary operators on $L^2(G)$), which is more elegant
and leads to the following conventions:
%


For   $\sigma \in {S_0{ '}(G)}$ (we will call them \textit{distributions} on $G$),
we define some operators (cf. \cite{fe89-2}):
\begin{gather*}
\langle T_x\sigma ,f\rangle:= \langle\sigma ,T_{-x}f\rangle \quad
\text{for} \; x\in G;\\
\langle g\ast \sigma ,f\rangle:
=\langle\sigma , {g}\checkm \ast f \rangle \quad
\text{for} \; g \in L^1(G);\\
\langle h\sigma ,f \rangle:=\langle\sigma ,hf\rangle \quad
\text{for} \; h\in A(G);\\
\langle \hat{\sigma},f\rangle := \langle\sigma ,\hat{f}\rangle;\\
\langle {\sigma}\checkm ,f\rangle:=\langle\sigma ,{f}\checkm \rangle,
  \quad \text{where} \,  {f}\checkm(z) = f(-z).
\end{gather*}
%
Recall also that a  distribution $\sigma \in {S_0{ '}(G)}$
is called \textbf{positive}, if
\[
f\geq 0\:\Longrightarrow \:\langle\sigma ,f\rangle\geq 0.
\]

Next we summarize a few properties of tensor products.
We do not start from the abstract definition given in the
literature on Banach spaces (although this would be possible
in principle) but take a rather naive approach to tensor
products. Let us start with some  simple terminology:

Let $f,g$ be functions on $G_{1}$ and $G_{2}$,
respectively. We use the  \textbf{tensor product} symbol
$f\otimes g$ for the description of a (separable) function on $G_{1}\times G_{2}$ given by
\[
f \otimes  g(x,y) := f(x)\cdot g(y);\quad x \in  G_{1},\:y \in  G_{2}.
\]
Let $B_{1}$ and $B_{2}$ be two Banach spaces, which are
continuously embedded
into $C^{b}(G_{1})$ and $C^{b}(G_{2})$.
The \textbf{projective tensor
product} of $B_{1}$ and $B_{2}$ is defined as
\[
B_{1} \hat{\otimes } B_{2} := \left\{ f \, \Big{|} \, f = \sum^{\infty
}_{n=1}f_{n} \otimes g_{n} , \; \sum^{\infty }_{n=1}\| f_{n}\| _{B_{1}}\|
g_{n}\| _{B_{2}} < \infty \right\}.
\]
$B_{1} \hat{\otimes } B_{2}$ is a Banach space continuously
embedded in $C^{b}(G_{1}\times G_{2})$:
\[
\| f\|_{\hat{\otimes}}:= \inf  \left\{  \sum^{\infty
}_{n=1}\|
f_{n}\| _{B_{1}}\| g_{n}\| _{B_{2}} \text{ with }f
= \sum^{\infty }_{n=1}f_{n}\otimes g_{n} \right\}.
\]

It is one of the key properties of the functor
$S_0$ (assigning ${S_{0}(G)}$ to each LCA group $G$)
that direct products of groups are
mapped on the tensor product of spaces, i.e.,  that
we have the following \textit{tensor product
property} which in turn is the basis for the all
important kernel theorem for operators:
\[
S_0(G_1 \times G_2) = S_{0}(G_{1}) \hat{\otimes}  S_{0}(G_{1})
\]
(cf.\ \cite{cofelu08,lo83-1},[see Theorem~\ref{ch18:thm7}.(d) in \cite{fe81-2}]).

For the functor $G \mapsto C_0(G)$ a similar property is not valid, i.e.,
the projective  tensor product
$V_{0}(G_{1}\times G_{2}) := C_{0}(G_{1}) \hat{\otimes} C_{0}(G_{2})$
is an interesting space different from $C_0(G_{1}\times G_{2})$.
One has the following proper inclusions
\[
S_0(G_{1}\times G_{2})  \subseteq
V_{0}(G_{1}\times G_{2})  \subseteq
C_{0}(G_{1}\times G_{2}),
\]
with dense embeddings in each case.

The dual of $V_{0}(G_{1}\times G_{2})$ therefore contains
$M(G_1 \times G_2)$ as a proper subspace.
It~is called the space of \textbf{bimeasures}
$BM(G_{1}\times G_{2})$.
For the properties of bimeasures, we refer to \cite{grsc84},
but they are all elements of $S'_0(G_{1}\times G_{2})$
and hence they have Fourier transforms.

A Radon measure $\mu$
(i.e., an element of the dual of $\mathcal{C}_c(G)$) is
called \textbf{translation bounded} if for any
$k \in \mathcal{C}_c(G)$ the
$\sup_{x \in G} |\mu(T_x(k))|$ is finite.
Equivalently, $\mu$ extends to a bounded linear functional
on the Wiener algebra $W(C_0,L^1)$, which in turn
contains ${S_{0}(G)}$ as a dense subspace, and which is the
minimal Segal algebra with pointwise $C_0(G)$ structure
(see \cite{fe77-3}).

A detailed discussion of the following concepts is given in
\cite{fe84}. A bounded set $S \subseteq  M(G)$  is called \textbf{tight} if
for every $\epsilon  > 0$  there exists
  $k\in \mathcal{C}_c(G)$ such that $\Vert k\cdot \mu  - \mu \Vert_M
 \le  \epsilon $  for all $\mu \in {}$S.

A bounded, $L^1$-tight net $(e_{\gamma })_{\gamma \in \varGamma }$  in
$L^{1}(G)$ is called a bounded approximate unit for
$L^{1}(G)$ (viewed as Banach algebra with convolution
which is denoted by the symbol ``$\ast$'')
if one has for any $f \in L^1(G)$:
$\lim_{\gamma } \Vert  e_{\gamma }\ast {}f -
f \Vert _{1} \rightarrow  0. $
The (distributional) limit of $e_{\gamma}$ (to be understood in the ${w}^*$-sense)
is Dirac's delta distribution denoted by $\delta_0$.

A net $(\mu _{\gamma })_{\gamma \in \varGamma }$ in $M(G)$
 is \textbf{vaguely} convergent with limit $\mu _{0}$ if one has $\lim_{\gamma } \mu _{\gamma}(k) = \mu _{0}(k) \ \forall \, k \in \mathcal{C}_c(G).$
If the net is bounded at least in the ${S_0{ '}(G)}$ sense, then
vague convergence is equivalent to ${w}^*$-convergence within
${S_0{ '}(G)}$ because $\mathcal{C}_c(G) \cap {S_{0}(G)} = A_c(G)$ is dense
in the Banach space ${S_{0}(G)}$.
Since ${S_{0}(G)} \subset L^1(G)$ (densely) one also
has a natural embedding of $C^b(G)$ into ${S_0{ '}(G)}$, via
$ h \to \sigma_h$ with $\sigma_h(f) = \int_G f(x) h(x) dx$.
If $h \in C^b(G)$ defines a functional  $\sigma \in {S_0{ '}(G)}$
in this way, we say that the \textit{regular distribution}
$\sigma$ is \textit{represented by} $h$.

\section{Definitions}
\label{ch18:sec1}

\begin{definition}\label{ch18:def1}
Let $\mathcal{H}$ be an arbitrary Hilbert space.

A measurable mapping $X:G \mapsto \mathcal{H}$ is called a \textbf{stochastic process} on G.
\end{definition}

\begin{definition}\label{ch18:def2}
A bounded linear mapping $\rho :{S_{0}(G)} \mapsto \mathcal{H}$ is called a
\textbf{generalized stochastic process (GSP)}.
\end{definition}

We can think of $\mathcal{H}$ as a Hilbert space of
$\mathbb{C}$-valued random
variables with zero expectation on an arbitrary probability space
$(\varOmega ,\varSigma ,P)$, that is, $L^{2}_{0}(\varOmega ,\varSigma ,P)$, but we
will only use the Hilbert space properties and therefore we write $\mathcal{H}$.

$(.| .)$ will denote the (sesquilinear) inner
product and $\| .\| _{\mathcal{H}}$ the norm in $\mathcal{H}$.

The following properties of GSPs will be of importance in this paper.
These definitions are of course analogous to the corresponding concepts
for classical stochastic processes in the following sense:
If a classical stochastic process is interpreted as a
generalized stochastic process (via integration), the classical
terminology is compatible with the one given here.

\begin{definition}\label{ch18:def3}\
\begin{enumerate}
\item[(a)] A GSP is called (wide sense time) \textbf{stationary} if
\[
(\rho (f)| \rho (g)) = (\rho (T_{x}f)| \rho (T_{x}g)) \ \forall \, x
\in G, \ \forall \, f,g \in {S_{0}(G)}.
\]

\item[(b)] A GSP is called (wide sense) \textbf{frequency stationary} if
\[
(\rho (f)| \rho (g)) = (\rho (M_{\chi}f)| \rho (M_{\chi}g)) \ \forall \, \chi
\in \hat{G}, \ \forall \, f,\,g \in {S_{0}(G)}.
\]

\item[(c)] A time-stationary and frequency-stationary process is called
\textbf{white noise}.
\end{enumerate}
\end{definition}

\begin{definition}\label{ch18:def4}\
\begin{enumerate}\leftskip5pt
\item[(a)] A GSP $\rho $ is called \textbf{bounded} (resp.\
  $\|\cdot\|_\infty$-bounded)
  if $\rho $ is bounded with respect to $\| .\| _{\infty }$:
\[
\exists \:  c > 0 \text{ such that }
\| \rho (f)\| _{\mathcal{H}} \le c\| f\| _{\infty } \
\forall \, f \in {S_{0}(G)}.
\]
\item[(b)] A GSP $\rho $ is called $\textbf{V}$\textbf{-bounded}
 (variation bounded) if:
\[
\exists \:  c > 0\text{ such that }\| \rho (f)\| _{\mathcal{H}} \le
c\| \hat{f}\| _{\infty } \ \forall \,  f \in {S_{0}(G)} .
\]
\end{enumerate}
\end{definition}
This is---as many other terms---not just randomly chosen
terminology by in accordance with established literature
in the stochastic literature, as are the following terms:

\begin{definition}\label{ch18:def5}
A GSP is called \textbf{orthogonally scattered} if
\[
\operatorname{supp}(f) \cap  \operatorname{supp}(g) = \emptyset \quad
\text{implies} \quad \rho (f)\perp \rho (g) \quad
\text{for} \,\,f,g\in {S_{0}(G)}.
\]

Due to the tensor product property of $S_{0}:
{S_{0}(G)}\hat{\otimes }{S_{0}(G)} = S_{0}(G \times G)$ [cf.~\cite{fe81-2} Theorem~\ref{ch18:thm7}.(d)] it is justified to
hope that the following definition
determines an element of $S_{0}(G\times G)'$.
\end{definition}

\begin{definition}\label{ch18:def6}
Let $\rho $ be a GSP.  The \textbf{autocovariance} (or
autocorrelation) distribution ${\sigma_{ \rho}}$ is defined as
$\langle\sigma _{\rho},f\otimes g\rangle := (\rho (f)| \rho (\bar{g}))
\ \forall \, f,g \in {S_{0}(G)}$.
\end{definition}

A priori the $
{\sigma_{ \rho}} $ functional is only defined for functions
of the form $f\otimes g, \, f,g \in {S_{0}(G)}$,
but we want ${\sigma_{ \rho}}$ to be a
distribution of two variables that is a linear functional on
${S_{0}(G)}\hat{\otimes }{S_{0}(G)} = S_{0}(G\times G)$.
Therefore we must
extend the definition of ${\sigma_{ \rho}}$ to functions $h:=
\sum^{\infty }_{n=1}f_{n}\otimes g_{n}$ .
This can be done by first
defining ${\sigma_{ \rho}}$ in the obvious way for finite sums.
It is not  difficult to show that this definition makes sense and that
${\sigma_{ \rho}}$ is bounded.
Then we can use the fact that the completion of
the space of all finite sums is equal to
${S_{0}(G)}\hat{\otimes }{S_{0}(G)}$ to extend
$\sigma _{\rho}$ (uniquely) to ${S_{0}(G)}\hat{\otimes
}{S_{0}(G)} = S_{0}(G\times G)$; hence $\sigma _{\rho}$
is an element of $S'_{0}(G\times G)$.



\section{Relations Between a GSP and Its Covariance}
\label{ch18:sec2}

As in the classical case, the relations between the properties
of a GSP $\rho $ and the covariance ${\sigma_{ \rho}}$ associated with
$\rho $ are important.  The following theorems will deal with this
aspect (see also \cite{ni75-1} for some aspects concerning
bimeasures):

\begin{theorem}\label{ch18:the1}
For a GSP $\rho $ the
following properties are pairwise equivalent:
\begin{enumerate}\leftskip4pt
\item[(a)] $\rho$  is stationary if and only if ${\sigma_{ \rho}}$ is
diagonally invariant,
i.e.  $ T_{(x,x)} {\sigma_{ \rho}} =
{\sigma_{ \rho}} \ \forall \, x \in G; $

\item[(b)] $\rho $ is bounded if and only if ${\sigma_{ \rho}}$
extends in a unique way to a bimeasure on $G\times G$;

\item[(c)] $\rho $ is orthogonally scattered 
if and only if
${\sigma_{ \rho}}$ is supported by the diagonal, i.e.,
 $ \operatorname{supp}({\sigma_{ \rho}}) \subseteq \varDelta_{G}
 := \left\{ (x,x) \mid x\in {}G \right\},$
and this holds if and only if
there exists a positive and translation
bounded measure $\tau_{\rho}$ with
\[
\langle{\sigma_{ \rho}},f\otimes g\rangle = \langle\tau_{\rho},fg\rangle
\ \, \, \forall \, f,g \in {S_{0}(G)}.
\]
\end{enumerate}
\end{theorem}

\begin{proof}\
\begin{enumerate}\leftskip4pt
\item[(a)] Follows from the definitions.

\item[(b)] $(\Rightarrow )$ By the continuity of ${\sigma_{ \rho}}$, it
follows from the definition  of  a bounded GSP that
the following holds:
\[
| \langle{\sigma_{ \rho}}, h \rangle| = 
| \langle{\sigma_{ \rho}},\sum^{\infty }_{n=1}f_{n}\otimes
g_{n}\rangle| \le \sum^{\infty }_{n=1}| (\rho (f_{n})| \rho
(\bar{g}_{n}))| \le c^{2} \sum^{\infty }_{n=1}\| f_{n}\| _{\infty }\|
g_{n}\| _{\infty }
\]
for all admissible representations
$\sum^{\infty}_{n=1}f_{n}\otimes g_{n}$
of $h\in {}S_{0}(G\times G)$;

\noindent hence
$| \langle{\sigma_{ \rho}},h\rangle|  \le
 c^{2}\|h\| _{V_{0}}  \ \forall \,  h \in S_{0}(G\times G)$.
The density of $S_{0}(G \times G)$ 
in $V_{0}(G \times G)$ now implies that
${\sigma_{ \rho}}$ extends to a uniquely determined
bimeasure on $G \times G$.

$(\Leftarrow )$ Boundedness of $\rho $ (with respect
   to the sup-norm) follows from the estimate
\[
\| \rho (f)\| ^{2}_{\mathcal{H}} = (\rho (f)| \rho (f)) = \langle\sigma
_{\rho},f\otimes \bar{f}\rangle \le c\| f\otimes \bar{f}\| _{V_{0}} \le
c\| f\| ^{2}_{\infty }
\]

\item[(c)] (First equivalence $\Leftarrow )$ Follows directly from the definitions.
\end{enumerate}
(First equivalence $\Rightarrow )$ 
First it is clear
that $\langle{\sigma_{ \rho}},f\otimes g\rangle = 0$ whenever
$\bigl(\operatorname{supp}(f)\times \operatorname{supp}(g) \bigr) \cap \varDelta _{G}  =
 \emptyset$.
Making use of the tensor product property for $S_0$, i.e., using $S_{0}(G\times G) = {S_{0}(G)}\hat{\otimes }{S_{0}(G)}$
and suitably refined partitions of unity (in both factors)
one derives therefrom that
$\langle{\sigma_{ \rho}},h\rangle = 0$ for any
$h \in S_{0}(G\times G)$ having compact support disjoint to
$\varDelta _{G}$.
This implies $\operatorname{supp}(\sigma _{\rho}) \subseteq \varDelta _{G}$.

(Second equivalence $\Rightarrow$) $\varDelta _{G}$ being a set of
\textit{spectral synthesis} (cf.  \cite{re68} Chap.7, Theorem~4.1, and Chap.6
Remark~1.5) a distribution ${\sigma_{ \rho}}$ with support on $\varDelta
_{G}$ satisfies ${\sigma_{ \rho}}(F) = {\sigma_{ \rho}}(H)$
if the two restrictions to $\varDelta _{G}$ are equal, i.e.,
if ${\operatorname{Restr}_{\varDelta_{G}}}(F) = {\operatorname{Restr}_{\varDelta_{G}}} (H)$ [recall that $ {\operatorname{Restr}_{\varDelta_{G}}}$, given by  ${\operatorname{Restr}_{\varDelta_{G}}}(F)(x) = F(x,x)$
maps $S_{0}(G \times G)$ onto
$S_{0}(\varDelta _{G})$ by \cite{fe81-2},
Theorem~\ref {ch18:thm7}.(c)].  The representation of
$\sigma _{\rho}$ by $\tau_{\rho}$ follows therefrom
using the canonical identification $j_{G}$ of $G$
and $\varDelta _{G}$ (given by $j_{G}(x) = (x,x)$) and the formula
\[
\langle{\sigma_{ \rho}},f\otimes g\rangle =
\langle\tau_{\rho},
[{\operatorname{Restr}_{\varDelta_{G}}}(f\otimes g)] \circ j_{G} \rangle =
\langle\tau_{\rho},f g \rangle.
\]
To show that $\tau_{\rho}$ is positive, we take a net
$f_{\alpha}  \in  {S_{0}(G)}$ with $|f_{\alpha}|^2  \rightarrow
\delta_0$ and define
\[
\langle ( \tau_{\rho})_{\alpha},g \rangle
= \langle \tau_{\rho} \ast
|f_{\alpha}|^2,g \rangle = \langle \tau_{\rho},
|(f_{\alpha}|^2) \checkm  \ast g \rangle.  
\]
It is clear that $\langle (\tau_{\rho})_{\alpha},g \rangle \rightarrow
\langle \tau_{\rho},g \rangle  \ \forall \,  g  \in  {S_{0}(G)}$. In
addition $(\tau_{\rho})_{\alpha}$ can be identified with the bounded
function $h_{\alpha}(x) =
\langle \tau_{\rho}, T_x (|f_{\alpha}|^2) \checkm \rangle
 = \langle \tau_{\rho}, |T_{-x} {f_{\alpha}\checkm}|^2 \rangle$.

As $\langle \tau_{\rho}, f \bar{f}\rangle
= (\rho(f)|\rho(f))  \geq  0  \ \forall \,  f  \in {S_{0}(G)}$
it is obvious that $h_{\alpha}(x)  \geq  0  \ \forall \,  x  \in
 G$. This implies
 $\langle (\tau_{\rho})_{\alpha},g \rangle  \geq    0
 \ \forall \,  g  \in  {S_{0}(G)}, \: g \geq 0$, and thus
$\tau_{\rho}$ is positive.  
But the positive elements
of ${S_0{ '}(G)}$ are translation bounded measures (cf. \cite{fe80}
Prop.\ B4, or \cite{ho89} Appendix: Theorem~2.3).
Since the opposite direction is obvious the proof is complete.\qed
\end{proof}

\setcounter{corollary}{1}
\begin{corollary}\label{ch18:cor2}
A GSP $\rho $ is bounded and
orthogonally scattered
if and only if there exists a
bounded measure $\mu_{\rho}$ on $G$ such that
\[
\langle{\sigma_{\rho}},f\otimes g\rangle = \langle\mu_{\rho},fg\rangle
= \int^{}_{G}f(x)g(x)d\mu_{\rho}(x) \quad
\forall \, f,g \in {S_{0}(G)}.
\]
\end{corollary}

\begin{proof}
$(\Leftarrow )$ Follows from Theorem~\ref{ch18:the1}.(c) and~(b).

$(\Rightarrow )$  Theorem~\ref{ch18:the1}.(c)  implies the first part of the formula for $\tau_{\rho}\in {}{S_{0}(G)}. \: \rho $
being bounded, and $\tau_{\rho}$ is a bimeasure (which is supported by the diagonal).  To prove that $\tau_{\rho}$ is a bounded measure, we have to
show that $\tau_{\rho}$ is bounded with respect to $\| .\| _{\infty }$.

It is possible to write $f\in {}C_{0}(G)$ as $f =
f_{1}f_{2}$ with $f_{1},f_{2}\in {} C_{0}(G)$ and \\
$\| f\| _{\infty } = \| f_{1}\| _{\infty
}\| f_{2}\| _{\infty }$  (e.g.\ $f_{1}(x):= \arg
(f(x)) \sqrt{| f(x) |}, \,
f_{2}(x):=  \sqrt{| f(x) |}$).

\newpage  
Now the following holds:
\begin{align*}
 | \langle\tau_{\rho},f\rangle| &=
 | \langle\tau_{\rho},{\operatorname{Restr}_{\varDelta_{G}}} (f_{1}\otimes f_{2})\rangle| =
 | \langle\sigma _{\rho},f_{1}\otimes f_{2}\rangle| \\&\le \|
\sigma _{\rho}\| _{BM}\| f_{1}\| _{\infty }\| f_{2}\| _{\infty }
= c\|f\| _{\infty }. \quad \quad  \qed 
\end{align*}
\end{proof}


\section{The Spectral Process}
\label{ch18:sec3}

The following definition is analogous to  the  definition  of  the
Fourier transform for distributions.

\begin{definition}\label{ch18:def7}
Given a GSP $\rho $ on $G$, we define a
GSP $\hat{\rho}$ on $\hat{G}$:
\[
\hat{\rho}(f) := \rho (\hat{f})  \
\forall \,  f \in S_{0}(\hat{G}).
\]
$\hat{\rho}$ is called the \textbf{spectral process} to $\rho $.
\end{definition}


By viewing $G$ as the dual group of $\hat{G}$ (via
Pontrjagin's theorem) we can also apply the inverse
Fourier transform (viewed now as the inverse to
the mapping from $S_{0}(\hat{G})$ to $S_{0}(G)$).

\begin{definition}
\label{ch18:def8}
Given a GSP $\rho $ on $G$, we define a
GSP $\check{\rho}$ on $\hat{G}$:
\[
 \check{\rho}(f) := \rho (\check{f}) \quad  \
\forall \,  f \in S_{0}(\hat{G}),
\]
where the inverse Fourier transform for test functions
is defined via
\[
 \check{f}(x) := h(x), \quad \mbox{for} \quad f = \hat{h} , \,
 h \in S_{0}(\hat{G}).
\]
\end{definition} 

\par
\medskip
 Since the mappings $f \mapsto \check{f}$ (inverse Fourier transform)
and $f \mapsto \hat{f}$ (forward Fourier transform)
define isomorphisms between ${S_{0}(G)}$ and $S_{0}(\hat{G})$, respectively, it is clear that $\check{\rho}$ and $\hat{\rho}$ are GSPs.


The reader should observe that the good properties of the functor
 $S_0$ and consequently of the functor  ${S'_0}$
gives the Fourier transform and the inversion (often called
``spectral representation'' of the process) a more symmetric
description. This is in sharp contrast to the usual setting,
where often stochastic processes are
viewed  as Hilbert space valued measures, or vector measures,
i.e.,  as linear operators
from $\mathcal{C}_c(G)$ into  $\mathcal{H}$, and the Fourier transform
resp.\ the inverse Fourier transform are defined in a rather
different setting, requiring additional technical tools and
arguments. On the other hand one may add that even for the
deterministic case ($\mathcal{H} = \mathbb{C}$) the use of the
Schwartz-Bruhat space over $G$ is much more cumbersome
than the usual Schwartz space over $R^d$. In fact, we hope
that the reader will appreciate the technical simplicity
of the description offered here, while still providing
arguments at the natural level of generality.

The following lemma contains some simple facts about these operators:
\newpage  

\setcounter{lemma}{2}
\begin{lemma}\label{ch18:lem3}
For a GSP $\rho $ the following properties are pairwise equivalent:
\begin{enumerate}\leftskip5pt
\item[(a)]
$\rho $ resp.  $ \hat{\rho} $ is bounded $\Longleftrightarrow \hat{\rho} $ resp.  $\rho $ is V-bounded;

\item[(b)] $\rho $ resp. $ \hat{\rho} $
is stationary $\Longleftrightarrow \hat{\rho} $ resp.  $\rho $ is
frequency stationary;

\item[(c)] $\rho = \hat{\tau} \Longleftrightarrow
\check{\tau} = \hat{\rho}$ for some GSP $\tau$ on $\hat G$.
\end{enumerate}
\end{lemma}

\begin{proof}
This follows directly from the definitions.\qed
\end{proof}

From Lemma~\ref{ch18:lem3}.(c) it is clear that
$\rho \mapsto ({\rho}\checkm) \hat{\;} =
{\hat{\rho}}\checkm $ is the
inverse mapping of $\rho \mapsto \hat{\rho}$.
This shows that the Fourier transform is a bijective
mapping between the GSPs over $G$ and
those over $\hat{G}$.

\setcounter{theorem}{3}
\begin{theorem}\label{ch18:the4}
Let $\rho $ be a GSP, then we have:
\begin{enumerate}
\item[(a)] $\langle\hat{\sigma}_{\rho},f\otimes g\rangle =
\langle\sigma_{\hat{\rho}},f\otimes {g}\checkm\rangle.$

\item[(b)] $\rho \text{ is orthogonally scattered if and only if }
 T_{(t,t)}\sigma_{\hat{\rho}} = \sigma_{\hat{\rho}} \
 \forall \, \, t \in \hat{G}.$
\end{enumerate}
\end{theorem}

\begin{proof}\
\begin{enumerate}\leftskip4pt
\item[(a)] Making use of the isometric $L^1$-involution
 $ g \mapsto g^*: g^*(t) = \bar{g}\checkm =(g \checkm)^-$,
with the useful property that
 $ \widehat{g^*} = \bar{\widehat{g}},$ we have:\\
$\langle\hat{\sigma }_{\rho},f\otimes g\rangle
= \langle\sigma _{\rho},\hat{f}\otimes \hat{g}\rangle
= (\rho (\hat{f})| \rho (\bar{\widehat{g}}))
= (\hat{\rho}(f)| \hat{\rho}({g^*}))
= \langle\sigma_{\hat{\rho}},
f\otimes\bar{{g^*}} \rangle
= \langle\sigma_{\hat{\rho}},f\otimes g\checkm \rangle$. \par

\item[(b)] Follows from Theorem~\ref{ch18:the1}.(c) and the fact that  $\sigma$ is
$H$-invariant if and only if
$\operatorname{supp}(\hat{\sigma }) \subseteq  H^{\perp}$
(cf. \cite{fe80} 
Theorem ~B2.iii) and part a) above.
\end{enumerate}
\end{proof}

\setcounter{corollary}{4}
\begin{corollary}
\label{ch18:cor5}
A GSP  $\rho $  is V-bounded if and only if $  \hat{\sigma }_{\rho}$
extends to a bimeasure.
\end{corollary}

\begin{proof}
Apply Lemma~\ref{ch18:lem3}.(a), Theorem~\ref{ch18:the1}.(b), and Theorem~\ref{ch18:the4}.(a). \qed
\end{proof}

\section{Characterization of Stationary Processes}
\label{ch18:sec4}

\begin{corollary}\label{ch18:cor6}
For a GSP $\rho $ the following properties hold: 
\begin{enumerate}\leftskip5pt
\item[(a)] $\rho $ is  frequency stationary if and only if $\rho $ is
orthogonally scattered;

\item[(b)] $\rho $ is stationary if and only if there exists a positive translation bounded measure $\tau_{\hat{\rho}}$ on 
 $G$ with $\langle\sigma_{\hat{\rho}},f\otimes g\rangle =
\langle\tau_{\hat{\rho}},fg\rangle \ \forall \, f,g \in {S_{0}(G)};
\tau_{\hat{\rho}}$ is called the \textbf{spectral measure} of $\rho $.
\end{enumerate}
\end{corollary}

\begin{proof}\
\begin{enumerate}\leftskip5pt
\item[(a)] Follows from Lemma~\ref{ch18:lem3}.(b), Theorem~\ref{ch18:the1}.(a), and
Theorem~\ref{ch18:the4}.(b).

\item[(b)] Follows from Lemma~\ref{ch18:lem3}.(b), part a, and
Theorem~\ref{ch18:the1}.(c).
\end{enumerate}
\end{proof}

\begin{rem}
Because of Corollary~\ref{ch18:cor6}.a, it is clear that $ \rho $ is white noise if and only if $ \rho $ is
stationary and orthogonally scattered.
\end{rem}

\begin{rem}
In the classical theory of stochastic processes, Corollary~\ref{ch18:cor6}.(a)
is called the ``spectral representation of a stationary process'' (cf.
\cite{do90} p.~527 or \cite{gisk74} p.~244).  The spectral measure
is called power spectrum or---if it has a continuous density---spectral
density of the process.

The following theorem
contains a characterization of the covariance of stationary GSPs.  The
necessary part is a kind of ``existence theorem'' for stationary GSPs.
\end{rem}

\setcounter{theorem}{6}

\begin{theorem}
\label{ch18:thm7}
\noindent For $\sigma \in
{}S'_{0}(G\times G)$ the following two properties are equivalent:

\noindent $\sigma $ is covariance of a stationary GSP $\Longleftrightarrow \sigma$ is diagonally invariant
and positive definite.
\end{theorem}

\begin{proof}
$(\Rightarrow )$ The invariance was already shown in Theorem~\ref{ch18:the1}.(a).  By Corollary~\ref{ch18:cor6}.(b) it is clear that the covariance distribution of
$\hat{\rho }$ is positive and this implies [by Theorem~\ref{ch18:the4}.(a)] that
$\hat{\sigma }$ is positive, which is equivalent to the positive
definiteness of $\sigma$.

\noindent $(\Leftarrow )$ $\sigma$ being
diagonally invariant and positive definite it follows
that $\hat{\sigma}$  is supported by
$\varDelta _{G}^\perp = \nabla \hat{G} :=
\left\{ (t| -t) , t\in {}\hat{G} \right \}$  
(cf.\cite{fe80}, 
Thm.B2.iii) and that
$\langle\hat{\sigma },f\otimes \bar{f} \checkm \rangle \ge 0
\ \forall \, f \in S_{0}(\hat{G})$
(as $f\otimes \bar{f} \checkm $ is
non-negative on $\nabla \hat{G}$).

\noindent This implies
$\langle\sigma,\hat{f}\otimes \widehat{f^*}   \rangle
 \ge 0 \ \forall \, f \in S_{0}(\hat{G})$
 and this is, using again $ \widehat{g^*} = \bar{\widehat{g}},$
equivalent to
$$
\langle\sigma ,  h \otimes \bar{h}\rangle \ge  0
\ \forall \, \quad  h \in
S_{0}(G).
$$

We have proved that the form $Q(f,g) := \langle\sigma
,f\otimes \bar{g}\rangle $ defines a positive semi-definite sesquilinear
form on $S_{0}(G)\times S_{0}(G)$.  Since $N = \left\{ f \mid
\langle\sigma
,f\otimes \bar{f}\rangle = 0 \right\} $ is a linear subspace of
$S_{0}(G)$ $Q$ defines a canonical inner product on $\mathcal{H}_{1}:=
S_{0}(G)/N$.  The Hilbert space obtained by completion can be denoted by
$\mathcal{H}$.  It is then clear that the canonical projection, followed by
the embedding of $\mathcal{H}_{1}$ into $\mathcal{H}$, defines a bounded,
linear mapping $\rho $ from $S_{0}(G)$ into $\mathcal{H}$, i.e., it is a
GSP.  Of course $\sigma $ coincides with $\sigma _{\rho }$.  By the
diagonal invariance of $\sigma $ the stationarity of $\rho $ follows.
\qed
\end{proof}

\setcounter{corollary}{7}

\begin{corollary}
\label{ch18:cor8}
Let $\sigma
\in {}S'_{0}(G\times G)$:

\noindent Then $\sigma $ is the covariance of
an orthogonally scattered GSP $\rho$
\begin{align*}
& \Longleftrightarrow \: \exists \:  \tau \text{ positive and translation
bounded with: }
\langle\sigma ,f\otimes g\rangle\\ & =
\int^{}_{G} fg d\tau  \ \forall \,  f,g \in S_{0}(G)
\end{align*}
\end{corollary}

\begin{proof}
$(\Rightarrow )$ Has been shown in Theorem~\ref{ch18:the1}.(c).

$(\Leftarrow )$ Let $\omega \in {}S'_{0}(\hat{G}\times
\hat{G})$ be defined in the following way:
\[
\langle\omega ,f\otimes g\rangle := \langle\hat{\sigma },f\otimes
\check{g}\rangle = \langle\sigma ,\hat{f}\otimes \check{g}
\hat{\;}\rangle = \int^{}_{G} \hat{f} \check{g} \hat{\;}d\tau \ \forall
\, f,g \in S_{0}(\hat{G}).
\]

\noindent Then $\omega $  is  diagonally  invariant  and  positive
definite.  By Theorem~\ref {ch18:thm7} the existence of a stationary GSP $\hat{\rho} $
over $\hat{G}$ with covariance $\omega $ follows.  Hence, by Theorem~\ref{ch18:the4}.(a),
$\rho $ is a GSP over $G$ with covariance $\sigma $, which is
orthogonally scattered in view of Lemma~\ref{ch18:lem3}.(b) and Corollary~\ref{ch18:cor6}.(a). \qed
\end{proof}

The following theorem characterizes white noise
(cf.  Definition~\ref{ch18:def3}.(c) in three different ways.

\setcounter{theorem}{8}
\begin{theorem}
A GSP $\rho $ is white noise if and only if one
of the following (equivalent) conditions is satisfied:
\begin{enumerate}\leftskip5pt
\item[(a)] $\exists \: c \ge 0\text{ such that }(\rho (f)| \rho (\bar{g})) = c\int^{}_{G} f(t)g(t)dt \ \quad \forall \, f,g \in {S_{0}(G)}; $

\item[(b)] $\| \rho (f)\| _{\mathcal{H}} = c\| f\| _{2}$ for
some $c \ge 0 $ and $\ \, \, \forall \, f \in {S_{0}(G)}$
(i.e., $\rho $ is
a scalar multiple of an isometry between $L^{2}(G)$ and $\mathcal{H})$.
\end{enumerate}
\end{theorem}

\begin{proof}

(a) This follows from  Theorem.1.(a) and (c) and
Corollary~\ref{ch18:cor6}.(a) together with the uniqueness of the Haar measure as a
positive translation invariant measure (cf.  \cite{re68} Chap. 3,3.1 and
references there).

(b) $\Rightarrow $ white noise) By means of the polar decomposition,
$(\rho (f)| \rho (\bar{g}))$ can be expressed with the help of
$\| \rho(f)\| _{\mathcal{H}}=c\| f\| _{2}$
and $\| \rho (\bar{g})\| _{\mathcal{H}} = c\| g\| _{2}$.
Thus it is clear that $(\rho (f)| \rho(\bar{g}))$
is not changed by modulation or translation.

(White noise $\Rightarrow$ b) Follows from part a.\qed
\end{proof}

\section{Relations Between GSPs and Other Theories}
\label{ch18:sec5}

It is clear that it is not possible to associate arbitrary
GSPs with (ordinary) stochastic processes, as there are GSPs with a covariance
distribution which cannot be represented by an ordinary function. But
we shall prove that any GSP with a covariance distribution
induced by some continuous bounded function in $C^{b}(G\times G)$
can be identified with a
uniquely determined stochastic process in the classical sense. On the other
hand any mean square continuous stochastic process can be identified with a
uniquely determined GSP. The exact formulation of this fact is contained in
the following theorem.

\begin{theorem}\label{ch18:the10}
Let $\rho $ be a GSP with covariance ${\sigma_{ \rho}}$:
\begin{enumerate}\leftskip5pt
\item[(a)] If $\sigma _{\rho} \in  S'_{0}(G\times G)
\text{is represented by some }
h \in  C^{b}(G\times G)$, then
 $\rho (f_{\alpha })$ is a Cauchy net in $\mathcal{H}$ whenever
$(f_{\alpha })_{\alpha \in I}$ is a vaguely convergent, $L^{1}$-bounded, and tight net in ${S_{0}(G)}$.

\item[(b)] In the above situation $\rho $ extends to
 a bounded linear operator
$\tilde{\rho}: M(G) \mapsto \mathcal{H}$ , which is $\sigma $-norm
continuous on tight subsets of $M(G)$, meaning that it converts
vaguely convergent, bounded, and  tight nets $(\mu_\alpha)$
into norm convergent nets $\tilde{\rho}(\mu_\alpha)$ in $\mathcal{H}$.
In particular,
$\lim_{y\rightarrow x} \tilde{\rho}(\delta _{y}) = \tilde{\rho}(\delta
_{x})$ in $\mathcal{H}$.  If $\left\{ \rho (f), f \in {S_{0}(G)} \right\} $ is dense in $\mathcal{H}$,
the extension $\tilde{\rho}$ is uniquely determined.

\item[(c)] The mapping $\rho_G \colon x \mapsto \tilde{\rho}(\delta_x)$ is a bounded, continuous stochastic process on $G$ and
$h(x,y) := (\tilde{\rho}(\delta_x) | \tilde{\rho}(\delta_y))$
is the covariance function of $\rho_G$.

\item[(d)] For any continuous and bounded stochastic process
 $\rho_{1}:G \mapsto  \mathcal{H}$,
the covariance function of $\rho _{1}  $ given by
$h(x,y) := (\rho_{1}(x)| \rho _{1}(y))$
is bounded and continuous on $G\times $G.  By vector-valued
integration, $\rho _{1}$ may be lifted to a bounded linear mapping
$\tilde{\rho}_{1}: M(G) \mapsto  \mathcal{H},$ which is $\sigma$-norm
continuous on bounded tight subsets. By
way of restriction to ${S_{0}(G)}$, $\tilde{\rho}_{1}$ may be considered
as a GSP  and $h$
represents the covariance distribution of this GSP.
\end{enumerate}
\end{theorem}

\begin{proof}\
\begin{enumerate}\leftskip5pt
\item[(a)] Assume now that $(f_\alpha)$ is a net in ${S_{0}(G)}$, which  is a vaguely
convergent, bounded, and tight in $M(G)$.
To prove that the net
$\rho(f_{\alpha})_{\alpha \in I}$ is a Cauchy net
in $\mathcal{H}$, we use the fact that for any $k \in \mathcal{C}_c(G)$
one has
\begin{gather*}
\| \rho (f_{\alpha }){-}\rho (f_{\beta })\| ^{2}_{\mathcal{H}}
{=} (\rho (f_{\alpha }-f_{\beta })| \rho (f_{\alpha }-f_{\beta })) {=}
\langle\sigma _{\rho},(f_{\alpha }{-}f_{\beta })\otimes (\bar{f}_{\alpha
}{-}\bar{f}_{\beta })\rangle {=} \\
{=} \langle\sigma _{\rho} \cdot k \otimes k,(f_{\alpha }{-}f_{\beta
})\otimes
(\bar{f}_{\alpha }-\bar{f}_{\beta })\rangle + \langle\sigma _{\rho},k(f_{\alpha }-f_{\beta })\otimes (1-k)(\bar{f}_{\alpha
}-\bar{f}_{\beta })\rangle + \\
+ \langle\sigma _{\rho},(1-k)(f_{\alpha }-f_{\beta })\otimes
(\bar{f}_{\alpha }-\bar{f}_{\beta })\rangle.
\end{gather*}
As $f_{\alpha}$ is tight and $L^{1}$-bounded,
there exists $k \in  \mathcal{C}_c(G)$ with
$\| (1-k)$\break $(f_{\alpha } - f_{\beta })
\| _{1}< \epsilon  \ \forall \,  \alpha  \in I$. Thus
the second of the three terms can be handled in the following way:
\begin{gather*}
| \langle\sigma _{\rho},k(f_{\alpha }-f_{\beta })\otimes
(1-k)(\bar{f}_{\alpha }-\bar{f}_{\beta })\rangle| =
\\
{=} \left| \int^{}_{G} \left( \int^{}_{G}
h_{\rho}(x,y)k(x)(f_{\alpha }(x){-}f_{\beta }(x)) dx \right) (1{-}k(y))(\bar{f}_{\alpha }(y){-}\bar{f}_{\beta
}(y)) dy \right| \le
\\
\le  \| h_{\rho}\| _{\infty
} \int^{}_{G}    \left( \int^{}_{G} | k(x)(f_{\alpha }(x){-}f_{\beta }(x))|  dx \right) |
(1{-}k(y))(\bar{f}_{\alpha }(y)-\bar{f}_{\beta }(y))|  dy \le \\
\le  \| h_{\rho}\| _{\infty } \, C \int^{}_{G} |
(1-k(y))(\bar{f}_{\alpha }(y)-\bar{f}_{\beta }(y))|  dy
\end{gather*}
and the last integral can be made arbitrarily small by suitable
choice of $k$ above, \textit{independent} of $\alpha$ and $\beta$.
The third term can be treated in the same way as the second.
As the first term  converges to  0 if $f_{\alpha}$
is  vaguely convergent, it follows that
$\| \rho (f_{\alpha}) - \rho (f_{\beta})\| _{\mathcal{H}}$
tends to zero and this implies
that $(\rho (f_{\alpha }))_{\alpha \in I}$
is a Cauchy net in the norm topology.

\item[(b)] As any measure in $M(G)$ can be represented as the
$w^*$-limit of a bounded, tight net of functions in $S_0(G)$,
the density  of
$\left\{  \rho (f) , f \in {S_{0}(G)} \right\} $
in $\mathcal{H}$ implies the existence of a uniquely
determined element denoted by
$\tilde{\rho}(\mu ) \in  \mathcal{H}$,
such that $(\tilde{\rho}(\mu )| \rho (g)) := \lim_{\alpha }(\rho
(f_{\alpha })| \rho (g))  \ \forall \,  g \in {S_{0}(G)}$. The notation is
justified because it is independent of the choice of the net
$(f_{\alpha })$ with $\lim_{\alpha} f_{\alpha} = \mu$.
This can be seen using the proof of part a)
with two different $L^1$-bounded, tight nets with the
same vague limit.

Now we show the continuity of $\tilde{\rho}:$
Let $(\mu _{\beta })_{\beta \in J}$  be a bounded and tight
net in $M(G),\,  w^\ast$-convergent  with  limit $\mu$.
Then we find that for any ${w}^*$-neighborhood
$  U := U(k_{1},k_{2},\ldots ,k_{n},\epsilon ) \in  \mathcal{U},$
 with $ k_{i} \in  {S_{0}(G)}  $, one can find for every $\beta$  some  $ f_{(U,\beta)} \in  {S_{0}(G)}$
such that the following holds:
\begin{gather*}
| \langle\mu _{\beta },k_{i}\rangle - \langle f_{(U,\beta)},k_{i}\rangle|
< \epsilon /2, \quad \forall \,  i=1\ldots n, \ \forall \, \beta\\
\text{and }\quad \quad | \langle\mu _{\beta },k_{i}\rangle - \langle\mu
,k_{i}\rangle|  < \epsilon /2,  \quad \forall \; \beta  > \beta _{0}.
\end{gather*}
Combining these two estimates one obtains
\[
| \langle \mu ,k_{i}\rangle - \langle
f_{(U,\beta)},k_{i}\rangle|  <  \epsilon,
 \quad \forall \,  i=1\ldots n \ \forall \,  \beta  > \beta _{0}.
\]
We have shown: $(f_{(U,\beta)})$ is a
$w^\ast$-convergent, $L^{1}$-bounded, and
tight net with the same limit
$\mu $ and $f_{(U,\beta)} \in {S_{0}(G)}$, for any
$U \in \mathcal{U}$.
Part a) implies that $\rho(f_{(U,\beta)})$ converges in
norm and
the limit is called $\tilde{\rho}(\mu )$ according to the above
definition.  Due to the construction of $f_{(U,\beta)}$, it is easy to
see that
$\tilde{\rho}(\mu _{\beta })$ converges to $\tilde{\rho}(\mu )$ as
well and this shows the ${w}^*$-to-norm continuity of $\tilde{\rho}$
on bounded, tight subsets. Furthermore, this implies the norm-norm
continuity and thus the boundedness of the mapping $\tilde{\rho}$
follows.

\item[(c)] The continuity of
$\rho _{G}$ follows from the ${w}^*$-to-norm continuity of
$\tilde{\rho}$ on tight subsets.
The boundedness of $\rho _{G}$ follows
from the fact that $\tilde{\rho}$ is bounded with respect to $\| .
\|_M$ together with
$\| \delta_x \|_M \,=\,1 \, \ \forall \, \, x \in G$.

 Let $(f_{\alpha })_{\alpha \in I}$  be tight, $L^{1}$-bounded
and  vaguely  convergent  with
limit $\delta _{0}$ (i.e., a generalized ``Dirac sequence'');
then the following completes the proof of part~c):
\begin{gather*}
h(x,y) = \int^{}_{G\times G} h(t) (\delta _{x}\otimes \delta _{y}) dt =
\lim_{\alpha } \langle h , T_{x}f_{\alpha }\otimes T_{y}f_{\alpha }
\rangle =\\
= \lim_{\alpha } \langle \sigma _{\rho} , T_{x}f_{\alpha }\otimes
T_{y}f_{\alpha } \rangle = (\tilde{\rho}(\delta _{x})| \tilde{\rho}(\bar{\delta }_{y})) = (\tilde{\rho}(\delta _{x})| \tilde{\rho}(\delta _{y})).
\end{gather*}

\item[(d)] The estimate
$ |h(x,y)| = |(\rho_{1}(x)| \rho _{1}(y))| \leq 
\| \rho_{1}(x)\|_{\mathcal{H}}\|\rho_{1}(y)\|_{\mathcal{H}} \le  c^{2}$
proves that $h$ is bounded. The continuity of $h$ results from the
continuity of $\rho _{1}$ and the inner product.
\end{enumerate}

With the help of vector-valued integration we can define
$ \tilde{\rho}_{1}(\mu)$ in the weak sense. In fact, for any
$ l \in \mathcal{H}$, the function $ (\rho _{1}(x)| l)$ is bounded and continuous;
hence the following integral makes sense:
\[
\bigl( \tilde{\rho}_{1}(\mu )| l\bigr) :=
\int^{}_{G} (\rho_{1}(x)| l) d\mu \quad \text{for} \,\, l \in \mathcal{H}\text{ and }\mu \in M(G).
\]
In view of the  Riesz  representation  theorem, $\tilde{\rho}_{1}(\mu )$  is  a  well-defined element of $\mathcal{H}$,
as $|(\tilde{\rho}_1(\mu)|l)| \: \leq \: c  \| \mu \|_M
\| l \|_{\mathcal{H}}$. The last inequality implies
$\| \tilde{\rho}_1(\mu) \|_\mathcal{H} \: \leq \: c \|
\mu \|_M$ and thus the boundedness of $\widetilde{\rho_1}$ with
respect to $\| . \|_M$.

To prove the ${w}^*$-to-norm continuity, we take a bounded, tight,
$w^\ast$-convergent net $\mu_{\alpha}$ in $M(G)$ with limit $\mu$. As
the mapping $x \mapsto ({\rho_1}(x)|l)$ is continuous and bounded for
any $l \in \mathcal{H}$ and $\mu_{\alpha}$ is tight we get it is enough
to know that the integrand is uniformly bounded and uniformly convergent
over compact subset of $G$, in order to conclude that
\begin{align*}
\lim_{\alpha}(\tilde{\rho}_1(\mu_{\alpha})|l) & = \lim_{\alpha} \int_G({\rho}_1(x)|l) d\mu_{\alpha}\\
& = \int_G({\rho}_1(x)|l) d\mu
 = (\tilde{\rho}_1(\mu)|l),
\end{align*}
which shows the ${w}^*-{w}^*$ continuity of $\tilde{\rho}_1$.
The convergence of
$\|\tilde{\rho}_1(\mu_{\alpha}) \|$
is shown by the following equality:
\begin{align*}
\lim_{\alpha} (\tilde{\rho}_1(\mu_{\alpha}) | \tilde{\rho}_1(\mu_{\alpha})) & = \lim_{\alpha }\int_G \int_G
(\rho_1(x)|\rho_1(y)) d\mu_{\alpha} d\mu_{\alpha}\\
& = \int_G \int_G (\rho_1(x)|\rho_1(y)) d\mu d\mu \quad
  = \, \,  \|\tilde{\rho}_1(\mu)\|^2,
\end{align*}
which is true as $(\rho_1(x)|\rho_1(y))$ is continuous and
bounded and $\mu_{\alpha}$ is $w^\ast$-convergent,
bounded, and tight. The last result together with the
${w}^*$-to-norm continuity of $\tilde{\rho}_1$
implies the ${w}^*$-to-norm continuity, and our claim is proved.
Furthermore 
we get the required GSP by restriction to $ {S_{0}(G)}$:
$\rho  := \tilde{\rho}_{1}| _{S_{0}}$.

It remains to be shown that
$h(x,y) = (\rho _{1}(x)| \rho_{1}(y))$
represents  the covariance distribution
$\sigma _{\rho}$ of $\rho $. This follows from the identity
\begin{gather*}
\langle \sigma _{\rho},f\otimes g \rangle = \bigl( \rho (f)| \rho
(\bar{g})\bigr)  = \int^{}_{G} (\rho _{1}(x)| \rho (\bar{g})) f(x) dx \\ \quad\quad\quad\quad\quad\quad\quad=
\int^{}_{G} \int^{}_{G} g(y)(\rho _{1}(x)| \rho _{1}(y)) dy f(x) dx\\
=
\int^{}_{G} \int^{}_{G}h(x,y) f(x) g(y) dx dy
= \langle h,f\otimes g\rangle \text{ for }f,g \in {S_{0}(G)}.
\quad \quad \qed 
\end{gather*}
\end{proof}

\begin{rem}
As any measure in $M(G)$ can be represented as the
$w^\ast$-limit of a bounded tight net of functions in ${S_{0}(G)}$
(e.g.\ by convolving $\mu \in M(G)$ with a Dirac sequence
$f_\alpha \in {S_{0}(G)}$ and using the fact that ${S_{0}(G)} \ast M(G) \subset {S_{0}(G)}$)
or as a limit of a net of discrete (and bounded) measures, the ${w}^*$-to-norm continuity of the mappings $\tilde{\rho}$ and
$\widetilde{\rho_1}$ on tight, bounded subsets implies the uniqueness of these
extensions from ${S_{0}(G)}$ or $G$ to $M(G)$. Therefore Theorem~\ref{ch18:the10} describes a bijective identification between continuous bounded stochastic processes and GSPs with continuous bounded covariance.
\end{rem}

\setcounter{corollary}{10}
\begin{corollary}\label{ch18:cor11}
Let $\rho $ be a V-bounded GSP. Then one has:
\begin{enumerate}\leftskip5pt
\item[(a)] $\rho $ can be identified
with a uniquely determined stochastic process.

\item[(b)] $\rho $ extends to $M(G)$ and $\hat{\rho}$ to
$\mathcal{F}(M(G))$; therefore
\[
\rho (\mu ) = \hat{\rho}(h)\text{  if }\hat{\mu } = \check{h} ,
\text{ and in particular, }
\rho (\delta _{x}) = \hat{\rho}(\chi _{x}).
\]
\end{enumerate}
\end{corollary}

\begin{proof}
Corollary~\ref{ch18:cor5} says: $\rho$ is V-bounded
if and only if  $\hat{\sigma }_{\rho}$ extends to a bimeasure.

Since the Fourier transform of a bimeasure is a bounded,
continuous function
(cf.\ \cite{grsc84}, Theorem~2.4(i), and Definition~2.1),
the corollary follows from
\hbox{Theorem}~\ref{ch18:the10}.\qed
\end{proof}

\begin{rem}
The formula $\rho (\delta _{x}) =
\hat{\rho}(\chi _{x})$ in Corollary~\ref{ch18:cor11}.(b) can be seen as an alternative
formulation of
the ``representation theorem of V-bounded stochastic  processes  as
Fourier transforms of  stochastic  measures.''  That  it  is  actually
equivalent to Niemi's formulation in \cite{ni75}, p.~35, will become
clear by Proposition~\ref{ch18:pro12}.
\end{rem}

The two preceding theorems show the strong relations between  GSPs
and stochastic processes.  Together  with  the  conditions  stated
above the concepts are even equivalent. From this point of view it
is possible to see aspects of the classical theory in a new  light
and to apply mathematical methods to stochastic processes in a new
way. This new point of view leads of course to  new  and  in  many
cases short and clear proofs of classical theorems.  The  theorems
stated in the previous chapter were proven for GSPs.  If we add the
presupposition that the covariance distribution $\sigma _{\rho}$ is
represented by a bounded, continuous function, all facts are proven for
stochastic processes as well, as it is obvious that the definitions of
certain properties for GSPs and stochastic processes are the same.  We
will use this considerations in the sequel to prove some results on
V-bounded and harmonizable stochastic processes in a new~way.

The following proposition compares GSPs and vector measures as
defined by Niemi (cf.  \cite{ni75} p.~15), that are (with respect to
the inductive limit topology) continuous and linear mappings $\mu :
\mathcal{C}_c(G) \mapsto \mathcal{H}$.  Since $\mathcal{C}_c(G)$ and ${S_{0}(G)}$ are
not related by inclusions no general comparison is possible, but the
following holds:

\setcounter{proposition}{11}
\begin{proposition}\label{ch18:pro12}
Under the assumption of boundedness or V-boundedness, the concepts of vector
measures and GSPs are equivalent.
\end{proposition}

\begin{proof}
Because of the definition of boundedness and V-boundedness, it is obvious that a
bounded GSP extends to a bounded linear mapping from $C_{0}(G)$.
V-bounded ones extend to bounded linear mappings on $\mathcal{F}(C_{0}(\hat{G}))$.  Since $\mathcal{C}_c(G)$ as well as ${S_{0}(G)}$ are
dense in these spaces, the equivalence of both concepts follows.  \qed
\end{proof}

\begin{rem}
The concept of stationary GSPs is more general than that of
stationary vector measures (cf.  \cite{ho89}).
\end{rem}


\section{Harmonizable Generalized Stochastic Processes}
\label{ch18:sec6}

For GSPs it is also possible to define generalizations of the
concept of stationarity.  The first of the two different concepts is the
concept of V-boundedness (cf.~Definition~\ref{ch18:def4}.(b)), which was first  introduced by Bochner (cf.  \cite{bo56} p.18).
Using Corollary \ref{ch18:cor11}.(a), it is obvious that any V-bounded GSP can be
identified with a V-bounded stochastic process.

\begin{rem}
According to
\cite{ra82-1}, p.315, Theorem~4.2, our definition of V-boundedness for GSPs
(with values in a Hilbert space) is equivalent to \textbf{weak
harmonizability} as defined for the first time in \cite{ro59}.
\end{rem}

The following definition corresponds to the
definition given in \cite{ni75} p.~35 for stochastic processes:

\begin{definition}
\label{ch18:def9}
$A$ GSP is called (\textbf{strongly}) \textbf{harmonizable} if
$\hat{\sigma}_{\rho}$ can be identified with a bounded measure
(i.e., is a functional which is continuous with respect to
the \hbox{sup-norm}).
\end{definition}

\begin{proposition}
\label{ch18:pro13}
For GSPs  $\rho $  with covariance ${\sigma_{ \rho}}$ one has:
\begin{enumerate}\leftskip5pt
\item[(a)] $\rho $ is harmonizable if and only if
${\sigma_{ \rho}}$ belongs to the Fourier-Stieltjes
algebra $B(G\times G) := \mathcal{F}(M(\hat{G}\times \hat{G}))$;
furthermore one has the following identity:
\[
{\sigma_{ \rho}} = h_{\rho}(x,y) = \langle \nu ,\bar{\chi}_{x}
\otimes \bar{\chi }_{y}\rangle  \ \quad
\forall \,  x,y \in G\text{ and }\nu
\in  M(\hat{G}\times \hat{G}).
\]

\item[(b)] If $\rho$ is harmonizable, then  $\rho$ is V-bounded.

\item[(c)] For stationary  GSP $\rho $ we have the following equivalence:
\[
\rho\ \mathrm{is \, harmonizable}\ \Longleftrightarrow\ \rho\ \mathrm{is \,  V-bounded}.
\]
\end{enumerate}
\end{proposition}

\begin{proof}\
\begin{enumerate}\leftskip5pt
\item[(a)] $h_{\rho}$ is an element of the Fourier-Stieltjes
algebra   if and only if $\hat{\sigma }_{\rho}$ is a bounded measure on $\hat{G}\times
\hat{G} \Leftrightarrow \rho $ is  harmonizable.

\item[(b)] $\rho$ 
harmonizable $\Rightarrow \hat{\sigma}_{\rho}$ extends to a bounded
measure (that is a bounded mapping on $C_{0}(G\times G) ) \Rightarrow
\hat{\sigma}_{\rho}$ defines a  bimeasure  which
implies by Corollary~\ref{ch18:cor5} that $\rho $ is V-bounded.

\item[(c)] $(\Rightarrow )$ This has been shown already in part b).

$(\Leftarrow )$ Let $\rho $ be a stationary V-bounded GSP. It
follows that $\hat{\rho}$ is bounded and orthogonally scattered.
Corollary~\ref{ch18:cor2} shows that there exists a bounded measure $\mu_{\hat{\rho}}$ on $\hat{G}$ with
\[
\langle \sigma_{\hat{\rho}},f\otimes g\rangle = \langle \mu_{\hat{\rho}},fg\rangle.
\]
Therefore we can identify $\mu_{\hat{\rho}}$
(as a measure on the diagonal
$\varDelta (\hat{G})$) with $\sigma_{\hat{\rho}}$,
which is therefore a bounded measure .  Thus $\rho $ is
harmonizable. \qed
\end{enumerate}
\end{proof}

\begin{rem}
As the Fourier-Stieltjes algebra is a subset of $C^{b}(G)$, Proposition~\ref{ch18:pro13}.(a) together with Theorem~\ref{ch18:the10} shows that any harmonizable GSP can be
identified with a (strongly) harmonizable stochastic process; the
converse is trivial.
\end{rem}

Since there are stationary
GSPs having a covariance distribution which cannot be identified with a
continuous function, it is obvious that stationary GSPs need not be
harmonizable or V-bounded. But if we add the continuity of
${\sigma_{ \rho}}$, we get:
\begin{proposition}\label{ch18:pro14}
Let $\rho $ be a GSP with covariance ${\sigma_{ \rho}}$:
\[
\rho \text{ is  stationary}, {\sigma_{ \rho}} \in C^{b}(G\times G)
\Longrightarrow \rho \, \, \text{is  harmonizable}.
\]
\end{proposition}

\begin{proof}
$\rho $ stationary and ${\sigma_{ \rho}} \in
C^{b}(G)$ imply [cf. Theorem~\ref{ch18:thm7}] that ${\sigma_{ \rho}}$ can be identified
with a continuous positive definite function. It follows by Bochner's
theorem (cf.  \cite{ru62} p.19) that ${\sigma_{ \rho}}$ is in the
Fourier-Stieltjes algebra $B(G\times G) := \mathcal{F}(M(\hat{G}\times
\hat{G}))$.\qed
\end{proof}

\setcounter{corollary}{14}
\begin{corollary}\label{ch18:cor15}
For a continuous stochastic process $X$
one has the following chain of implications:
\[
X \; \text{is   stationary }\Longrightarrow  X\; \text{is\; harmonizable }
\Longrightarrow  X \; \text{is  V-bounded. }
\]
\end{corollary}

\begin{proof}
This claim follows from
Propositions~\ref{ch18:pro13}.(b) and~\ref{ch18:pro14}, recalling that any stationary
continuous stochastic process can be identified
with a stationary GSP with ${\sigma_{ \rho}} \in C^{b}(G\times G).$
\qed
\end{proof}

The following theorem (for stochastic processes it was first
proved in \cite{ni75}) states that any V-bounded GSP can be
approximated by harmonizable ones.  This is of interest as there are
GSPs which are V-bounded but not harmonizable.  For the proof we need
linear smoothing operators for GSPs which are defined in the same way as for distributions.

\begin{definition}\label{ch18:def10}
For a   GSP, $\rho$ on G  and $ f \in {S_{0}(G)},$ one defines:
\begin{enumerate}\leftskip5pt
\item[(a)] $h\rho (f) := \rho (hf) $ , for
 $h \in A(G) := \mathcal{F}\bigl( L^{1}(\hat{G})\bigr)$;

\item[(b)] $k\ast {}\rho (f) := \rho (\check{k}\ast {}f) $ for
$ k \in L^{1}(G)$.
\end{enumerate}
\end{definition}

\begin{rem}
Since $ f \mapsto hf $ and $ f \mapsto
k\ast f $ define linear and bounded operaters on $ {S_{0}(G)} $, it follows
that $h\rho $ and $k\ast {}\rho $ define GSPs.
\end{rem}

\setcounter{lemma}{15}
\begin{lemma}\label{ch18:lem16}
Let $\rho $ be a GSP on $G$, $k \in
L^{1}(G), \, h \in A(G)$. Then the following equations hold true:
\begin{enumerate}\leftskip5pt
\item[(a)] $(k\ast {}\rho ) \hat{\;} = \hat{k}\hat{\rho}$;

\item[(b)] $\sigma_{h\rho} = (h\otimes \bar{h})\sigma _{\rho} $;

\item[(c)] $\sigma_{k\ast {}\rho } = (k\otimes \bar{k})\ast {}{\sigma_{ \rho}}$.
\end{enumerate}
\end{lemma}

\begin{proof}
The easy calculations are left to the reader.
\end{proof}

\setcounter{theorem}{16}
\begin{theorem}\label{ch18:the17}
For any V-bounded GSP $ \rho $, there exists a net
$\bigl( \rho _{\eta }\bigr) _{\eta \in E}$ of harmonizable GSPs such
that the corresponding autocovariance functions converge
uniformly on compact sets, i.e.,   
\[
  \sigma_{\rho_{\eta }}(x,y) \rightarrow {\sigma_{ \rho}}(x,y)
 \quad  \text{ for } \eta \rightarrow \infty.
\]

\end{theorem}

\begin{proof}
Let $(e_{\alpha })_{\alpha \in I}$ be a net in
${S_{0}(G)}$ constituting a tight, $L^{1}$-bounded approximate
unit for $L^{1}(G)$, and let $(u_{\beta })_{\beta \in J}$
be a bounded approximate unit for the Fourier algebra
$A(G)$ in ${S_{0}(G)}$.  Then we
set $\rho_{\eta } := u_{\beta }(e_{\alpha }\ast {}\rho), \, \eta :=
(\alpha ,\beta ) \in E := I\times J$.
According to Lemma~\ref{ch18:lem16}.b and~\ref{ch18:lem16}.c, we have
\[
\sigma _{\rho_{\eta }} = (u_{\beta }\otimes \bar{u}_{\beta
})[(e_{\alpha }\otimes \bar{e}_{\alpha })\ast {\sigma_{ \rho}} ].
\]
Furthermore $d_{\alpha } := e_{\alpha }\otimes
\bar{e}_{\alpha }  $  is  a  tight, $L^{1}$-bounded
approximate unit for $L^{1}(G\times G)$, and $u_{\beta }\otimes
\bar{u}_{\beta } := v_{\beta }$ is a bounded approximate unit in
$A(G\times G)$.

Let $K$ be a given subset of $G\times G$.  We want to show that $\sigma _{\rho_{\eta }}$ converges uniformly on
$K$.  Writing $\| f\| _{K,\infty } := \sup _{x\in K} |f(x)|$, we have
to verify that for any $\epsilon > 0$
\[
\exists \: \eta _{0} \text{ such that } \| \sigma_{\rho_{\eta}} -
{\sigma_{ \rho}}\| _{K,\infty } \le \epsilon \quad
\forall \eta \ge \eta_{0}.
\]
We can use the following estimate:
\begin{align*}
\| \sigma_{\rho_{\eta}} - {\sigma_{ \rho}}\| _{K,\infty } &
= \| v_{\beta }(d_{\alpha }\ast {}{\sigma_{ \rho}}) -
{\sigma_{ \rho}}\|_{K,\infty}\\
& \le \| v_{\beta }(d_{\alpha }\ast {}{\sigma_{ \rho}}) -
d_{\alpha }\ast
{}{\sigma_{ \rho}}\| _{K,\infty } +\| d_{\alpha }\ast {}{\sigma_{ \rho}} -
{\sigma_{ \rho}}\| _{K,\infty } .
\end{align*}
It is not difficult to see that the second term of this
estimate tends to zero for $\alpha \rightarrow \infty $ and that the
first term converges to zero for $\beta \rightarrow
\infty $ for arbitrary fixed $\alpha $ , which proves the uniform
convergence of $\sigma_{\rho_{\eta}}$ over K.  On the other hand, we have $\sigma_{\rho_{\eta}} \in
S_{0}\ast {}S_{0}(G\times G) \subseteq S_{0}(G\times G)$ and this
implies that $\hat{\sigma }_{\rho _{\eta }} \in S_{0}(\hat{G}\times
\hat{G}) \subseteq M(\hat{G}\times \hat{G})$, showing that $\rho _{\eta
}$ is harmonizable $\forall \, \eta \in $ E.\qed
\end{proof}

\setcounter{corollary}{17}
\begin{corollary}\label{ch18:cor18}
Any continuous V-bounded stochastic process
$X$ can be approximated by harmonizable processes uniformly
over compact sets.
\end{corollary}

\begin{proof}
Any continuous V-bounded stochastic process can be
identified with a uniquely determined V-bounded GSP. Now we apply
Theorem~\ref{ch18:the17}. The approximating harmonizable GSPs can be identified with
harmonizable stochastic processes.
The following holds (cf.\ the proof of Theorem~\ref{ch18:the10}):
\[
\tilde{\rho}_{\eta }(\delta_{x}) =
\tilde{\rho}(u_{\beta}(e_{\alpha}\ast {}\delta_{x})) =
 \tilde{\rho}(u_{\beta}(T_{x}e_{\alpha})
 \rightarrow \tilde{\rho}(\delta _{x})
\]
uniformly on compact sets by the vague continuity
of $\tilde{\rho}$.\qed
\end{proof}

We conclude this paper with some remarks concerning the
dilation theory for GSPs.  As will be shown, our setting is also suitable
to describe the dilation theorem for stochastic processes.  More
precisely, we want to point out that any V-bounded GSP is the projection
of a stationary GSP (that means that there is a stationary dilation).
For stochastic processes this was first shown by H.  Niemi in
\cite{ni75-4} using the main result of \cite{ni77}.
We will use the invariance of GSPs under the Fourier transform
to obtain this theorem for GSPs as a direct corollary of the main result
of \cite{ni77}. Due~to the theory we have developed so far, the
dilation theorem for stationary stochastic processes is obtained as a
corollary as well.

\begin{definition}\label{ch18:def11}
Let $\mathcal{H} \subset \tilde{\mathcal{H}}$ be two
Hilbert spaces, $\rho $ a GSP with
$\left\{ \rho (f) \mid f {\in} {S_{0}(G)}
\right\}\bar{\;} = \mathcal{H}$.
A GSP $\tilde{\rho}$ into
$\tilde{\mathcal{H}}$ is called a \textbf{dilation} of
 $\rho $ if
\[
\rho (f) = P\bigl( \tilde{\rho}(f)\bigr) \
\forall \, f \in {S_{0}(G)} ,
\]
$P$ denoting the orthogonal projection from $\tilde{\mathcal{H}}$
into $\mathcal{H}$.
\end{definition}

\setcounter{theorem}{18}
\begin{theorem}\label{ch18:the19}
For any bounded GSP $\rho $, there exists a dilation $\tilde{\rho}$ into
$\tilde{\mathcal{H}}$ which is orthogonally scattered and bounded.
\end{theorem}

\begin{proof}
Due to the equivalence between bounded vector
measures and bounded GSPs (cf. Proposition~\ref{ch18:pro12}), we may refer to the proof
for vector measures which is given in \cite{ni77}, Theorem~13.\qed
\end{proof}

\setcounter{corollary}{19}
\begin{corollary}
\label{ch18:cor20}
Let $\rho $ be a GSP. Then
 $\rho$ is V-bounded  if and only if there exists a dilation
$\tilde{\rho}$ which is V-bounded and stationary.
\end{corollary}

\begin{proof}
$(\Rightarrow )\,\, \rho$ V-bounded $\Rightarrow \hat{\rho}$
bounded. Theorem~\ref{ch18:the19} implies: $\exists$ dilation $(\hat{\rho}) \tilde{\;} $ of $\hat{\rho}$
which is bounded and orthogonally scattered. It is easy to see that this
implies: $\tilde{\rho} := ((\hat{\rho})\tilde{\;}) \hat{\;}
\check{\;}$
which is V-bounded and stationary is a dilation of $\rho $.

$(\Leftarrow )$ As $\| \rho (f)\| _{\mathcal{H}} \le
\| \tilde{\rho}(f)\|_{\tilde{\mathcal{H}}}  \
\forall \,  f \in
 {S_{0}(G)}$ it is clear that a GSP  with
a V-bounded dilation is V-bounded itself. \qed
\end{proof}

As any continuous V-bounded stochastic process can be
identified with a V-bounded GSP and vice versa
the same result is proved for stochastic processes.  Since stationarity
implies V-boundedness for stochastic processes (cf. Corollary~\ref{ch18:cor15}) the
assumption of V-boundedness on the right-hand side can be omitted.

\begin{corollary}\label{ch18:cor21}
Let $X$ be a continuous stochastic process. Then
$X$  is V-bounded if and only if there exists
a stationary dilation of $X$.
\end{corollary}


\def\ackname{Acknowledgments}
\acknowledgement{
The author would like to thank
Goetz Pfander for providing access to his  series of new
papers prior to publication (during a conference in Argentina,
Summer 2013). The reader will find some additional information
in the PhD thesis of Wolfgang H\"ormann and a more recent
master thesis by Thomas Steger \cite{st09-4}.
The authors also have to
thank the editors of this volume for providing
the opportunity to show the gratefulness to Paul Butzer for his
influence on the first author, as a mathematician and as a person,
caring for science and for young(er) people.}



\end{document}